\documentclass[10pt]{amsart}

\usepackage{amsmath}
\usepackage{amssymb}
\usepackage{verbatim}
\usepackage{enumerate}

\usepackage{setspace}
\usepackage{fancyhdr}
\usepackage{indentfirst}
\usepackage{amsthm}
\usepackage{mathrsfs}
	
\theoremstyle{definition}
\newtheorem{thm}{Theorem}[section]
\newtheorem{lem}[thm]{Lemma}

\newtheorem{cor}[thm]{Corollary}

\newcommand{\bt}{\beta}
\newcommand{\gm}{\gamma}

\newcommand{\eps}{\varepsilon}
\newcommand{\zt}{\zeta}

\newcommand{\Z}{{\mathbb{Z}}}

\newcommand{\C}{{\mathbb{C}}}
\newcommand{\N}{{\mathbb{N}}}

\newcommand{\Aut}{{\mathrm{Aut}}}

\newcommand{\norm}[1]{\left\Vert#1\right\Vert}

\newcommand{\set}[1]{\left\{#1\right\}}
\newcommand{\ten}{\otimes}

\newcommand{\JS}{\mathcal{Z}}
\newcommand{\bigdsum}{\bigoplus}

\title[A Criterion for $\JS$-Stability]{A Criterion for $\JS$-Stability with 
Applications to Crossed Products}

\author{Julian Buck}

\begin{document}

\begin{abstract}
Building on an argument by Toms and Winter, we show that if $A$ is a 
simple, separable, unital, $\JS$-stable C*-algebra, then the crossed 
product of $C(X,A)$ by an automorphism is also $\JS$-stable, provided 
that the automorphism induces a minimal homeomorphism on $X$. As a 
consequence, we observe that if $A$ is nuclear and purely infinite then 
the crossed product is a Kirchberg algebra.
\end{abstract}

\maketitle 

\section{Introduction}

\indent

In \cite{JiangSu}, Jiang and Su constructed a C*-algebra $\JS$ 
(now known as the {\emph{Jiang-Su algebra}}) which is simple, 
separable, unital, infinite-dimensional, strongly self-absorbing (in the 
sense of \cite{TomsWinter2}), nuclear, and has the same Elliott invariant 
as the complex numbers $\C$. A separable C*-algebra $A$ is said 
to be {\emph{$\JS$-stable}} if there is an isomorphism $A \ten \JS \cong 
A$. The property of $\JS$-stability appears to be intimately connected to 
the question of whether or not a simple, separable, nuclear 
C*-algebra is classified by its Elliott invariant. 
(See \cite{ElliottToms,Toms3} for example.)\\

In this article, we prove the following:

\begin{thm}\label{BZStable}
Let $X$ be an infinite compact metric space and let $A$ be a simple, 
separable, unital, $\JS$-stable C*-algebra. Let $\bt \in \Aut(C(X,A))$, 
and suppose that the homeomorphism $\phi \colon X \to X$ induced by 
$\bt$ is minimal. Then the crossed product C*-algebra, 
$C^*(\Z, C(X,A), \beta)$ is also $\JS$-stable.
\end{thm}

By ``the homeomorphism induced by $\bt$,'' what is meant is the induced 
map on the primitive spectrum of $C(X,A)$ (which can obviously be identified 
with $X$).\\

Significant progress has been made in recent years on the 
classification of crossed product C*-algebras arising from 
minimal dynamical systems. Toms and Winter (\cite{TomsWinter3}) showed 
that crossed products of infinite, finite-dimensional metric spaces by minimal 
homeomorphisms have finite nuclear dimension and are $\JS$-stable, and 
consequently that, when the projections in the crossed product separate 
traces, the crossed products are classified by ordered K-theory. More recently 
Elliott and Niu (\cite{ElliottNiu}) have demonstrated that $\JS$-stability holds for such 
crossed products even when $X$ is infinite-dimensional, so long as the minimal 
dynamical system has mean dimension zero.\\

Not as much is known in the case for crossed products of C*-algebras of 
the form $C(X,A)$. Hua (\cite{Hua}) has shown that in the case where $X$ 
is the Cantor set, $A$ has tracial rank zero, and the automorphisms in the fibre 
direction are K-theoretically trivial, the resulting crossed product has tracial rank 
zero. Our result contributes to the understanding of these crossed products, 
which are further studied by the first-named author in \cite{Buck}. The main 
results there assume that $A$ is locally subhomogeneous, and hence our 
Corollary \ref{Kirchberg}, which deals with the case where $A$ is purely 
infinite, is independent of the conclusions in \cite{Buck}.\\

We would like to thank Andrew Toms and Wilhelm Winter for suggesting 
the main technical lemma of this paper as a method to prove Theorem 
\ref{BZStable}, and for other helpful comments.

\section{The proof}

The proof that certain crossed products here are $\JS$-stable is based on an 
argument of Toms and Winter that crossed products of $C(X)$ by minimal 
homeomorphisms are $\JS$-stable. This argument appeared (as Theorem 4.4) 
in a preprint version (\cite{TomsWinter}) of \cite{TomsWinter3}. In the 
published version, it was replaced by an indirect proof of this fact.\\

We have broken apart their argument into a more general criteria for 
$\JS$-stability (the following theorem), followed by an application to our 
crossed products.

\begin{lem}\label{ZStabilityFromSubalgs}
Let $A$ be a separable, unital C*-algebra. Define $c_{0}, c_{1/2}, 
c_{1} \in C([0,1])$ by 
\[
c_{0} (t) = \begin{cases} 0 & t \leq 3/4, \\ 1 & t = 1, \\ \mbox{linear} &
\mbox{else}. \end{cases} \hspace{0.5 in} c_{1} (t) = \begin{cases} 1 & 
t = 0, \\ 0 & t \geq 1/4, \\ \mbox{linear} & \mbox{else}. \end{cases}
\]

\[
c_{1/2} (t) = \begin{cases} 0 & t = 0,1, \\ 1 & 1/4 \leq t \leq 3/4, \\ 
\mbox{linear} & \mbox{else}. \end{cases}
\]
Suppose that for any finite set $\mathcal{F} \subset A$ and any $\eta 
> 0$, there exist $\JS$-stable subalgebras $A_{0}, A_{1/2}, A_{1} 
\subset A$ and a positive contraction $h \in A_{+}$ such that $A_{1/2} 
\subseteq  A_0 \cap A_1$ and for every $a \in \mathcal{F}$, there exist 
$a_{i} \in \overline{c_{i}(h) A_{i} c_{i}(h)}$ for $i=0,1/2,1$ such that
\begin{align*}
\norm{ a- (a_{0} + a_{1/2} + a_{1}) } &< \eta \text{ and} \\
\norm{ [a_{1/2},h] } &< \eta
\end{align*}
It then follows that $A$ is $\JS$-stable.
\end{lem}

\begin{proof}
Using \cite[Theorem 7.2.2]{RordamClassBook} and a diagonal sequence 
argument (cf.\ \cite[Section 4.1]{UnitlessZ}), it suffices to find, for every 
$\eps > 0$ and every pair of finite subsets $\mathcal{F} \subset A$ and 
$\mathcal{E} \subset \JS$, a unital $*$-homomorphism 
\[ 
\zt \colon \JS \to A_{\infty} := \prod_{\N} A / \bigdsum_{\N} A,
\]
such that $\norm{ [\iota_{A} (a),\zeta(z)] } < \eps$ for all $a \in \mathcal{F}$ 
and $z \in \mathcal{E}$ (with $\iota \colon A \to A_{\infty}$ the canonical 
embedding).

Therefore, let $\eps > 0$ and finite sets $\mathcal{F} \subset A$ and 
$\mathcal{E} \subset \JS$ be given. Define $d_{0}, d_{1/2}, d_{1}  \in 
C([0,1])$ by 
\[
d_{0} (t) = \begin{cases} 0 & t \leq 1/2, \\ 1 & t \geq 3/4, \\ \mbox{linear} &
\mbox{else}. \end{cases} \hspace{0.5 in} d_{1} (t) = \begin{cases} 1 & 
t \leq 1/4, \\ 0 & t \geq 1/2, \\ \mbox{linear} & \mbox{else}. \end{cases}
\]

\[ d_{1/2} (t) = \begin{cases} 0 & t = 0 \leq 1/4, t \geq 3/4, \\ 1 & t = 1/2, \\ 
\mbox{linear} & \mbox{else}. \end{cases}
\]
Then $\set{c_{0},c_{1/2},c_{1}}$ and $\set{d_{0},d_{1/2},d_{1}}$ are both 
partitions of unity for $[0,1]$. We then define 
\[
C = C^{*}(d_{0} \ten \JS \ten 1_{\JS} \ten 1_{\JS} \cup d_{1/2} \ten 1_{\JS} 
\ten \JS \ten 1_{\JS} \cup d_{1} \ten 1_{\JS} \ten 1_{\JS} \ten \JS) \subset 
C([0,1]) \ten \JS \ten \JS \ten \JS
\]
and
\[
\tilde{C} = C^{*}(C([0,1]) \ten 1_{\JS \ten \JS \ten \JS} \cup C).
\]
Identifying $C^*(d_0,d_{1/2},d_1)$ in the obvious way with $C(Y)$ where 
$Y=[\frac14,\frac34]$, we note that $C$ is a $C(Y)$-algebra, all of whose 
fibres are isomorphic to $\JS$. Therefore, by \cite{DadarlatWinter}, $C$ is 
$\JS$-stable, so there exists a unital $*$-homomorphism $\overline{\zeta} 
\colon \JS \to C \subset \tilde{C}$. Note that
\[ 
\mathcal{S} := (d_0 \ten \JS \ten 1_{\JS} \ten 1_{\JS}) \cup (d_{1/2} \ten 
1_{\JS} \ten \JS \ten 1_{\JS}) \cup (d_1 \ten 1_{\JS} \ten 1_{\JS} \ten \JS), 
\]
generates $C$ as a C*-algebra. So, by approximating $\mathcal E$ by 
$*$-polynomials in $\mathcal S$, we see that there exists $\beta > 0$ and a 
finite subset $\mathcal{S}' \subset \mathcal{S}$ such that, if $\psi \colon C 
\to B$ is any $*$-homomorphism between C*-algebras and 
$\norm{ [\psi(s),b] } < \beta$ for all $s \in \mathcal{S}'$ then 
$\norm{ [\psi(\overline{\zeta}(z)),b] } < \eps/2$ for all $z \in \mathcal{E}$.

Set $M = \max \set{\norm{z} \colon z \in \mathcal{E}}$ and let $\eta \leq 
\eps/(2M)$ be sufficiently small so that if $\norm{ [a_{1/2},h] } < \eta$ then 
$\norm{ [a_{1/2},d_i(h)] } < \beta$ for $i=0,1$. Use the hypothesis to find 
the subalgebras $A_{i}$ and the positive contraction $h$.

Since each $A_{i}$ is $\JS$-stable, there exists unital $*$-homomorphisms 
$\overline{\rho}_{i} \colon \JS \to (A_{i})_{\infty} \cap \iota_{A_i}(A_i) \subseteq 
A_{\infty} \cap \iota_A(A_i)$. Having found $\overline{\rho}_{1/2}$ first, a 
speeding up argument (cf.\ the proof of \cite[Proposition 4.4]{Winter:pure}) 
shows that we can arrange that $\overline{\rho}_i(\JS)$ commutes with 
$\overline{\rho}_{1/2}(\JS)$, for $i=0,1$.

We may define a unital $*$-homomorphism $\gm \colon \tilde{C} \to A_{\infty}$ 
by setting
\begin{align*}
\gm (f d_{0} \ten z \ten 1_{\JS} \ten 1_{\JS}) &= (f d_{0}) (h) 
\overline{\rho}_{0} (z), \\
\gm (f d_{1/2} \ten 1_{\JS} \ten z \ten 1_{\JS}) &= (f d_{1/2}) (h) 
\overline{\rho}_{1/2} (z), \\
\gm (f d_{1} \ten 1_{\JS} \ten 1_{\JS} \ten z) &= (f d_{1}) (h) 
\overline{\rho}_{1} (z),
\end{align*}
for all $f \in C([0,1])$. (The proof that this defines a $*$-homomorphism 
mainly consists of checking that anything occurring on the right-hand 
sides of two different equations above commutes.) Finally, define $\zeta 
= \gm \circ \overline{\zeta} \colon \JS \to A_\infty$.

For $a \in \mathcal{F}_1$, let $a \approx_\eps a_{0} + a_{1/2} + a_{1}$ as in 
the hypothesis. In fact, we may assume that $a_{i} = c_{i} (h) a_{i}' c_{i} (h)$ 
exactly, for some $a_{i}' \in A_{i}$. Then, for $z \in \mathcal{E}$,
\[ 
[a,\zeta(z)] \approx_{2 \norm{z} \eps} [a_{0},z] + [a_{1/2},z] + [a_{1},z]. 
\]
Notice that since $\overline{\zeta}(z) \in \tilde{C}$, it follows that there exists 
$z_{0} \in \JS$ such that
\[ 
\overline(\zeta) (z)(t) = z_{0} \ten 1_{\JS} \ten 1_{\JS} 
\]
for all $t \in [0,1/4]$. Consequently,
\begin{align*}
\zeta (z) a_{1} &= \gm (c_{1} \ten z_{0} \ten 1_{\JS} \ten 1_{\JS}) a_{1}' c_{1} 
(h) \\
&= c_{1} (h) \overline{\rho}_{1} (z_{0}) a_{1}' c_{1} (h) \\
&= c_{1} (h) a_{1}' \overline{\rho}_{1} (z_{0}) c_{1} (h) \\
&= a_{1} \zeta (z).
\end{align*}
(the last step is essentially done by reversing earlier steps). Likewise, $\zeta 
(z) a_{0} = a_{0} \zeta (z)$. Also, we have for $z \in \JS$,
\begin{align*}
[ a_{1/2}, \gm (d_{0} \ten z \ten 1_{\JS} \ten 1_{\JS}) ] &= [a_{1/2}, d_{0} (h) 
\overline{\rho}_{0} (z)] \\
&= [a_{1/2}, d_{0} (h)] \\
&< \bt.
\end{align*}
Likewise, we find that $\norm{ [a_{1/2}, \gm (s)] } < \beta$ for all $s \in 
\mathcal{S}'$, and therefore,
\[ 
\norm{ [a_{1/2},\gamma(z)] } < \eta/2 
\]
for $z \in \mathcal{E}$. It follows that
\[ 
\norm{ [a,\zeta(z)] } < \eps/2+\eps/2 = \eps,
\]
which completes the proof.
\end{proof}

\begin{proof}[Proof of Theorem \ref{BZStable}] 
Set $B := C^{*}(\Z, C(X,A), \beta)$, and let $u \in B$ denote the 
canonical unitary. We shall show that $B$ satisfies the hypotheses of 
Lemma \ref{ZStabilityFromSubalgs}. Let $\eta > 0$ and a finite set 
$\mathcal{F} \subset B$ be given. First, we may assume that 
\[
\mathcal{F} \subset \set{u^{j} C(X) \colon 0 \leq j \leq k-1}
\]
for some $k > 1$, since the linear span of these elements and their 
adjoints is dense in $B$. Set  $M := \max \set{\norm{f} \colon u^{j} f \in 
\mathcal{F}}$. Combining Propositions 1.1 and 3.2 of \cite{TomsWinter3}, 
there exists $h \in C(X) \ten 1_{A}$ and points $x_{0}, x_{1} \in X$ with 
disjoint orbits such that $h(x_{j}) = j$ (for $j=0, 1$) and such that $h,u$ 
satisfy the relations 
\[
\norm{ [u,h] } < \frac{\eta}{3M}, \hspace{0.5 in} \| [u,c_{i} (h)^{1/k}] \| 
< \frac{2 \eta}{3M k (k-1)}
\]
for $i=0,1/2,1$ and with the $c_{i}$ given as in the statement of 
Proposition \ref{ZStabilityFromSubalgs}. It then follows that, for $a = 
u^\ell f \in \mathcal{F}$ and $i=0,1/2,1$, we have 
\begin{align*}
\| c_{i} (h) a - c_{i} (h)^{(k-\ell)/k} f (c_{i} (h)^{1/k} u)^{\ell} \| &\leq 
M \| c_{i} (h)^{\ell/k}u^{\ell} - (c_{i} (h)^{1/k} u)^{\ell} \| \\
&\leq M \| [u,c_{i} (h)^{1/k}] \| (\ell + (\ell-1)+ \cdots + 1) \\
&< \frac{M \ell (\ell+1)}{2} \frac{2 \eta}{3M k (k-1)} \\
&\leq \eta/3.
\end{align*}
Thus,
\[ 
a = c_{0} (h) a + c_{1/2} (h) a + c_{1} (h) a \approx_{\eta} a_{0} 
+ a_{1/2} + a_{1}, 
\]
where $a_{i} = c_{i} (h)^{(k-\ell)/k} f (c_{i} (h)^{1/k} u)^\ell$.

Set $Y_{i}:=\{x_{i}\}$ for $i=0,1$ and $Y_{1/2}:=\{x_0,x_1\}$.
For $i=0,1/2,1$, set
\[ A_{i} := C^*(C(X,A) \cup uC_0(X \setminus Y_i,A)) \subseteq B. \]
We see that $A_{1/2} \subseteq A_{0} \cap A_{1}$, and $a_{i} \in 
\overline{c_{i} (h) A_{i} c_{i} (h)}$. Results of \cite{Buck} show that 
each $A_{i}$ is $\JS$-stable. Moreover, for $i=0,1/2,1$, (and in 
particular, for $i=1/2$),
\begin{align*}
\norm{ [a_{1/2},h] } &\leq \| [c_{i} (h) u^{\ell} f, h] \| + 
\frac{2 \eta}{3} \\
&\leq M \norm{ [u^\ell,h] } + \frac{2 \eta}{3} \\
&< \frac{M \eta}{3 M} + \frac{2 \eta}{3} \\ 
&= \eta.
\end{align*}
This verifies the hypotheses of Lemma \ref{ZStabilityFromSubalgs}, 
and therefore, $B$ is $\JS$-stable.
\end{proof}

It is well-known (see \cite{Rord}) that exact $\JS$-stable 
C*-algebras are either stably finite or purely infinite. This 
allows us to obtain a useful corollary in the case where $A$ is 
nuclear and purely infinite.

\begin{cor}\label{Kirchberg}
Adopt the hypotheses and notation of Theorem \ref{BZStable} 
and assume in addition that $A$ is nuclear and purely infinite. 
Then $B$ is a Kirchberg algebra. Consequently, $B$ has 
nuclear dimension at most 3.
\end{cor}

\begin{proof}
By Theorem \ref{BZStable}, the algebra $B$ is $\JS$-stable. It is 
clearly infinite, since it contains the purely infinite algebra $A$ as 
the subalgebra $1_{C(X)} \ten A$, and it is nuclear. Therefore it is 
purely infinite. Since $B$ is simple (this essentially follows from 
\cite{ArchSpiel}), it is a Kirchberg algebra. The conclusion about 
the nuclear dimension of $B$ follows immediately from Thereom 
7.1 of \cite{MatuiSato}.
\end{proof}

If $A$ is in the UCT class, then so is the algebra $B$ of Corollary 
\ref{Kirchberg}, and hence such algebras are classified by their 
K-theory using the theorems of Kirchberg and Phillips (\cite{Kirch}, 
\cite{NCP}).

\end{document}